\RequirePackage{fix-cm}
\documentclass[smallextended]{svjour3}       
\smartqed  
\usepackage{graphicx}
\usepackage{amsmath, amssymb}
\newcommand{\graph}{{\rm Gr}}

\newcommand{\X}{\mathbb{X}}

\newcommand{\A}{\mathbb{A}}
\newcommand{\K}{\mathbb{K}}
\newcommand{\N}{\mathbb{N}}
\newcommand{\Y}{\mathbb{Y}}
\newcommand{\Ss}{\mathbb{S}}
\newcommand{\R}{\mathbb{R}}

\newcommand{\oo}{{\mathcal{O}}}

\newcommand{\Gr}{{\rm Gr}}
\newcommand{\ilim}{\mathop{\rm lim\,inf}}
\newcommand{\slim}{\mathop{\rm lim\,sup}}

\begin{document}

\title{Continuity of Minima: Local Results\thanks{This research was
partially supported by NSF grant CMMI-1335296,
by the Ukrainian State Fund for Fundamental Research under grant
GP/F49/070, and by grant 2273/14 from the National Academy of
Sciences of Ukraine.}
}


\author{Eugene A. Feinberg         \and
        Pavlo O. Kasyanov 
}


\institute{E.A. Feinberg \at
              Department of Applied Mathematics and
Statistics,  Stony Brook University,
Stony Brook, NY 11794-3600, USA, \\
              Tel.: +1-631-632-7189\\
              \email{eugene.feinberg@stonybrook.edu}           
           \and
           P.O. Kasyanov \at
              Institute for Applied System Analysis, National Technical
University of Ukraine ``Kyiv Polytechnic Institute'', Peremogy ave., 37, build, 35, 03056, Kyiv, Ukraine,\ kasyanov@i.ua.
}

\date{Received: date / Accepted: date}

\maketitle

\begin{abstract}
This paper compares and generalizes Berge's maximum theorem for noncompact image sets established in Feinberg, Kasyanov and Voorneveld~\cite{FKV} and the local maximum theorem established in  Bonnans and Shapiro~\cite[Proposition~4.4]{Shapiro}.
\end{abstract}

%

\section{Introduction and Main Results} Let $\X$ and $\Y$ be
Hausdorff topological spaces, and $\Phi:\X\to\Ss(\Y)$ be a set-valued map, where $\Ss(\Y):=2^{\Y}\setminus \{\emptyset\}$ is
\textit{the family of all nonempty subsets} of the set $\Y.$  Consider  the graph of $\Phi,$ defined as
${\rm Gr}_\X(\Phi)=\{(x,y)\in \X\times\Y\,:\, y\in\Phi(x)\},$  and let $u:\Gr_\X(\Phi)\subseteq\X\times \Y\to
\overline{\mathbb{R}},$ where
$\overline{\mathbb{R}}:=\mathbb{R}\cup\{\pm\infty\}$ is \textit{the
extended real line}. Define
the value function
\begin{equation}\label{eq1}
v(x):=\inf\limits_{y\in {\Phi}(x)}u(x,y),\qquad 
x\in\X,
\end{equation}
and the solution multifunction
\begin{equation*}
{\Phi}^*(x):=\left\{y\in
{\Phi}(x):\,v(x)=u(x,y)\right\},\quad x\in\X.
\end{equation*}

To clarify the above definitions, consider a Hausdorff topological space $\A$.
For a nonempty set $A\subseteq \A,$  the notation $f:A\subseteq\A\to \overline{\R}$ means that for each $a\in A$ the value $f(a)\in\overline{\R}$ is defined.  In general, the function $f$ may be also defined outside of $A.$   The notation $f:\A\to \overline{\R}$ means that the function $f$ is defined on the entire space $\A.$  This notation is equivalent to the notation $f:\A\subseteq\A\to \overline{\R}, $ which we do not write explicitly.  For a function $f:A\subseteq\A\to \overline{\R}$ we sometimes consider its restriction $f:B\subseteq\A\to \overline{\R}$ to the set $B\subseteq A.$ Sometimes we consider functions with values in $\R$ rather than in $\overline{\R}.$

We recall that, for a
nonempty set $A\subseteq \A,$ a function $f:A\subseteq \A\to\overline{\R}$ is
called \textit{lower semi-continuous at $a\in A$}, if for each net
$\{a_i\}_{i\in I}\subset A$, that converges to $a$ in $\A$, the
inequality $\ilim_i f(a_i)\ge f(a)$ holds. A function
$f:A\subseteq \A\to\overline{\R}$ is called \textit{upper semi-continuous at
$a\in A$}, if $-f$ is lower semi-continuous at $a\in A$.
%
Consider the level sets
\[
\mathcal{D}_f(\lambda;A):=\{a\in A \, : \,  f(a)\le \lambda\},\qquad
\lambda\in\R.
\]
A function $f:A\subseteq \A\to\overline{\R}$ is called \textit{lower
semi-continuous on $A$}, if all the level sets
$\mathcal{D}_f(\lambda;A)$ are closed in $\A$. A function
$f:A\subseteq \A\to\overline{\R}$ is called \textit{inf-compact} (such a
function is sometimes called \emph{lower semi-compact}) \emph{on
$A$}, if all the level sets are compact in $\A$. We notice that, if
$A=\A$ and $f:\A\to\overline{\R}$, then $f$ is lower semi-continuous
at each $a\in\A$ if and only if it is lower semi-continuous on $\A$.
For an arbitrary nonempty subset $A$ of $\A$, lower semi-continuity
on $A$ implies lower semi-continuity at each $a\in A$, but not vice
versa. Indeed, if $A\subseteq\A$ is a nonempty set, $f:A\subseteq \A\to
\overline{\R}$ is lower semi-continuous on $A$, and $\{a_i\}_{i\in
I}\subset A$ converges to $a\in A$, then either $\ilim_i
f(a_i)=+\infty$ or there exists a subnet $\{a_j\}_{j\in J}\subseteq
\{a_i\}_{i\in I}$ such that $\{a_j\}_{j\in J}$ is eventually in $
\mathcal{D}_f(\lambda;A)$ for each real $\lambda > \ilim_i f(a_i)$.
Thus $f(a)\le \ilim_i f(a_i)$. Vice versa, if $\A=[0,1]$, $A=(0,1]$,
and $f\equiv 0$ on $A$, then $f$ is lower semi-continuous at each
$a\in A$, but it is not lower semi-continuous on $A$ because
$\mathcal{D}_f(1;A)=A$ is not closed in $\A$.

\begin{definition}{\rm (cf. Feinberg et al. \cite[Definition~1.3]{FKV}).}\label{def:glF-ic}
A function $u:{\rm Gr}_{\X}(\Phi)\subseteq \X\times\Y \to \overline{\mathbb{R}}$ is called
$\K\N$-inf-compact on ${\rm Gr}_{\X}(\Phi)$, if the following two
conditions hold:
\begin{itemize}
\item[(i)] $u:{\rm Gr}_{\X}(\Phi)\subseteq \X\times\Y\to \overline{\mathbb{R}}$ is lower semi-continuous on ${\rm
Gr}_{\X}(\Phi)$;
\item[(ii)] for any convergent net $\{x_i
\}_{i\in I}$ with values in $\X$ whose limit $x$ belongs to $\X$,
any net $\{y_i\}_{i\in I}$, defined on the same ordered set $I$ with
$y_i\in \Phi(x_i)$, $i\in I,$ and satisfying the condition that the
set $\{u(x_i,y_i): i\in I \}$ is bounded above, has an accumulation
point $y\in \Phi(x).$
\end{itemize}
\end{definition}

We remark that this definition is consistent with Feinberg et al. \cite[Definition~1.3]{FKV}, according to which  a function
$u:\X\times\Y \to \overline{\mathbb{R}}$ is called  $\K\N$-inf-compact on ${\rm Gr}_{\X}(\Phi)$, if it satisfies properties (i) and (ii) of Definition~\ref{def:glF-ic}.  A function $u: \X\times\Y \to \overline{\mathbb{R}}$ is $\K\N$-inf-compact on ${\rm Gr}_{\X}(\Phi)$ in the sense of \cite[Definition~1.3]{FKV} if and only if its restriction  ${\rm Gr}_{\X}(\Phi)$ is $\K\N$-inf-compact on ${\rm Gr}_{\X}(\Phi)$ in the sense of Definition~\ref{def:glF-ic}.   This is true because, if $u$ is also defined outside of ${\rm Gr}_{\X}(\Phi)$, its values at $(x,y)\notin {\rm Gr}_X(\Phi)$ do not affect the $\K\N$-inf-compactness of $u$ on ${\rm Gr}_{\X}(\Phi).$

For $Z \subseteq \X$ define the graph of a set-valued mapping
$\Phi:\X\to\Ss(\Y)$, restricted to  $Z$:
\[
{\rm Gr}_Z(\Phi)=\{(x,y)\in Z\times\Y\,:\, y\in\Phi(x)\}.
\]
The following definition introduces the notion of
$\K\N$-inf-compactness in a local formulation.

\begin{definition}\label{def:locF-ic}
Let $Z \subseteq \X$ be a nonempty set. A function $u: {\rm Gr}_{\X}(\Phi)\subseteq \X\times\Y  \to
\overline{\mathbb{R}}$ is called $\K\N$-inf-compact on ${\rm
Gr}_{Z}(\Phi)$, if the following two conditions hold:
\begin{itemize}
\item[(i)] $u:{\rm Gr}_{\X}(\Phi)\subseteq \X\times\Y\to \overline{\mathbb{R}}$ is lower semi-continuous at all $(x,y)\in{\rm
Gr}_{Z}(\Phi)$;
\item[(ii)] if a net $\{x_i
\}_{i\in I}$ with values in $\X$ converges to $x\in Z$, then each net
$\{y_i\}_{i\in I}$, defined on the same ordered set $I$ with $y_i\in
\Phi(x_i)$, $i\in I,$ and
satisfying the condition that the set $\{u(x_i,y_i): i\in I \}$ is bounded
above, has an accumulation point $y\in \Phi(x).$
\end{itemize}
\end{definition}
\begin{remark}
If $Z=\X$, then Definitions~\ref{def:glF-ic} and \ref{def:locF-ic}
of $\K\N$-inf-compactness on $\Gr_\X(\Phi)$ are equivalent. This
follows from the following statements: (a) conditions (ii) of
Definitions~\ref{def:glF-ic} and \ref{def:locF-ic} coincide, if
$Z=\X$; (b) conditions (i) and (ii) of Definition~\ref{def:locF-ic}
with $Z=\X$ imply condition (i) of Definition~\ref{def:glF-ic}; and
(c) condition (i) of Definition~\ref{def:glF-ic} yields condition
(i) of Definition~\ref{def:locF-ic}. Note that statement (b) holds
because, if $\lambda\in \R$ and a net $\{(x_i,y_i)\}_{i \in
I}\subset\mathcal{D}_u(\lambda;\Gr_\X(\Phi))$ converges to $(x,y)\in
\X\times\Y$, then condition (ii) of Definition~\ref{def:locF-ic}
yields that $y\in \Phi(x)$ and condition (i) of
Definition~\ref{def:locF-ic} implies that
$(x,y)\in\mathcal{D}_u(\lambda;\Gr_\X(\Phi))$. Therefore, the level
set $\mathcal{D}_u(\lambda;\Gr_\X(\Phi))$ is closed for each
$\lambda\in \R$. Statement (c) holds, because lower semi-continuity
on $\Gr_\X(\Phi)$ implies lower semi-continuity at each $(x,y)\in
\Gr_\X(\Phi)$.
\end{remark}
\begin{remark}
{\rm  If spaces $\X$ and $\Y$ are metrizable, then nets and subnets in the definition of $\K\N$-inf-compactness on ${\rm Gr}_{Z}(\Phi)$ can be replaced with sequences and subsequences; see~Lemma~\ref{lmmetr}.}
\end{remark}
 Note that a function $u:{\rm Gr}_{\X}(\Phi)\subseteq \X\times\Y\to \overline{\mathbb{R}}$ is {$\K\N$-inf-compact on
${\rm Gr}_{Z}(\Phi)$}, where $Z$ is a nonempty subset of $\X,$  if and only if it is $\K\N$-inf-compact on ${\rm
Gr}_{\{x\}}(\Phi)$ for all $x\in Z.$

For a Hausdorff topological space $\A$, we denote by $\K(\A)$ the family of
all nonempty compact subsets of $\A.$ A function $u:{\rm Gr}_{\X}(\Phi)\subseteq \X\times\Y\to \overline{\mathbb{R}}$ is called $\K$-inf-compact on ${\rm
Gr}_{\X}(\Phi)$, if for every $K\in \K(\X)$ this function is
inf-compact on ${\rm Gr}_K(\Phi)$; cf. Feinberg et al.
\cite[Definition 1.1]{Feinberg et al}.
{$\K\N$-inf-compactness on ${\rm Gr}_{\X}(\Phi)$} is a
more restrictive property than $\K$-inf-compactness property on ${\rm
Gr}_{\X}(\Phi)$; see Feinberg et al.
\cite[Theorem~2.1 and Example~5.1]{FKV}.
As shown in Feinberg et al.~\cite[Corollary~2.2]{FKV},
$\K$-inf-compactness and $\K\N$-inf-compactness on
${\rm Gr}_{\X}(\Phi)$ are equivalent, if $\X$ is  a compactly
generated topological space, and, in particular, if $\X$ is a metrizable topological space. Recall that a topological space $\X$
is \emph{compactly generated\/} (Munkres \cite[p. 283]{Munkres} or
a $k$-space, Kelley \cite[p.~230]{Kelley}, Engelking
\cite[p.~152]{Engelking}) if it satisfies the following property:
each set $A \subseteq \X$ is closed in $\X$ if $A \cap K$ is
closed in $K$ for each $K \in \K(\X)$. In particular, all locally
compact spaces (hence, manifolds) and all sequential spaces
(hence, first-countable, including metrizable/metric spaces) are
compactly generated;  Munkres \cite[Lemma 46.3, p.
283]{Munkres}, Engelking \cite[Theorem 3.3.20, p. 152]{Engelking}.

For a set-valued mapping ${F}:\X \to 2^\Y,$ let ${\rm Dom} F:=\{x\in
\X\,:\, F(x)\ne \emptyset\}$. Recall that a set-valued mapping
${F}:\X \to 2^\Y$ is \textit{upper semi-continuous} at $x\in {\rm
Dom}F$, if, for each neighborhood $\mathcal{G}$ of the set $F(x)$,
there is a neighborhood of $x$, say $\mathcal{O}(x)$, such that
$F(x^*)\subseteq \mathcal{G}$ for all $x^*\in \mathcal{O}(x).$ A
set-valued mapping ${F}:\X \to 2^\Y$ is \textit{lower
semi-continuous} at $x\in{\rm Dom} F$, if, for each open set
$\mathcal{G}$ with $F(x) \cap \mathcal{G} \neq \emptyset$, there is
a neighborhood of $x$, say $\mathcal{O}(x)$, such that if $x^*\in
\mathcal{O}(x)\cap {\rm Dom}F$, then $F(x^*)\cap
\mathcal{G}\ne\emptyset$. A set-valued mapping ${F}:\X \to \Ss(\Y)$
is called \textit{upper (lower) semi-continuous}, if it is upper
(lower) semi-continuous at all $x\in\X$. A set-valued mapping
${F}:\X \to \Ss(\Y)$ is called \textit{continuous}, if it is upper
and lower semi-continuous. A set-valued mapping ${F}:\X \to \Ss(\Y)$
is \textit{closed}, if ${\rm Gr}_\X(F)$ is a closed subset of
$\X\times\Y$.

For Hausdorff topological spaces, Berge's maximum
theorem for noncompact image sets
has the following formulation.

\begin{theorem}\label{th:BGL}{\rm (Berge's maximum theorem for noncompact image sets; Feinberg et al. \cite[Theorem~1.4]{FKV}).} If a function $u: \X \times \Y \to \mathbb{R}$ is
$\K\N$-inf-compact and upper semi-continuous on
$\graph_{\X}(\Phi)$ and $\Phi: \X \to \Ss(\Y)$ is a lower
semi-continuous set-valued mapping, then the value function $v:
\X\to \R$ is continuous and the solution multifunction
$\Phi^*:\X\to \K(\Y)$ is
upper semi-continuous and compact-valued.
\end{theorem}

In the classic Berge's maximum theorem \cite[p.~116]{Berge}, the function $u$ is assumed to be continuous and the set-valued mapping $\Phi$ is assumed to be continuous and compact-valued. As explained in Feinberg et al.~\cite{FKV},
these assumptions are more restrictive than the assumptions of Theorem~\ref{th:BGL}.

 Theorem~\ref{th:BGL} yields from the following three statements:
\begin{itemize}
\item[\textbf{(B1)}] (Lower semi-continuity of minima). If a function $u: \X \times \Y \to \overline{\mathbb{R}}$ is
$\K\N$-inf-compact on $\graph_{\X}(\Phi)$, then the value function $v:
\X\to \overline{\mathbb{R}}$ is lower semi-continuous; Feinberg et al. \cite[Theorem~3.4]{FKV}.
\item[\textbf{(B2)}] (Upper semi-continuity of minima). If a function $u: \X \times \Y \to \overline{\mathbb{R}}$ is
upper semi-continuous on $\graph_{\X}(\Phi)$ and $\Phi: \X \to \Ss(\Y)$ is a lower
semi-continuous set-valued mapping, then the value function $v:
\X\to\overline{\mathbb{R}}$ is upper semi-continuous; Hu and Papageorgiou \cite[Proposition~3.1, p.~82]{Hu}. \item[\textbf{(B3)}] (Upper semi-continuity of the solution multifunction). If a function $u: \X \times \Y \to \mathbb{R}$ is
$\K\N$-inf-compact on $\graph_{\X}(\Phi)$ and the value function $v:
\X\to \R$ is continuous, then the solution multifunction
$\Phi^*:\X\to \K(\Y)$ is
upper semi-continuous and compact-valued; Feinberg et al. \cite[p.~1045]{FKV}.
\end{itemize}


Continuity and semi-continuity of functions and multifunctions are local properties. Therefore, it is natural to
formulate the results on continuity of value functions and solution multifunctions in a local form. This is done in Bonnans and Shapiro \cite[Proposition~4.4]{Shapiro}.

\begin{theorem}\label{th: locprop1BS}{\rm (Bonnans and Shapiro \cite[Proposition~4.4]{Shapiro}).}
Let $u: \X \times \Y \to {\mathbb{R}}$ and $x\in \X$.
Supposed that:
\begin{itemize}
\item[(i)] the function $u$ is continuous on $\X \times \Y$;
\item[(ii)] the set-valued mapping $\Phi:\X\to\Ss(\Y)$ is closed;
\item[(iii)] there
exists $\lambda\in \R$ and a compact set $C\subset \Y$ such that, for
every $x^*$ in a neighborhood of $x,$ the level set
$\mathcal{D}_{u(x^*,\cdot)}(\lambda;\Phi(x^*))$ is nonempty and
contained in $C$;
\item[(iv)] for each neighborhood $\oo(\Phi^*(x))$ of
the set $\Phi^*(x)$, there exists a neighborhood $\oo(x)$ of $x$
such that $\oo(\Phi^*(x))\cap \Phi(x^*)\ne \emptyset$ for all
$x^*\in \oo(x)$.
\end{itemize}
 Then:
 \begin{itemize}
 \item[(a)] the function $v$ is continuous at  $x$,
 \item[(b)] the set-valued function $\Phi^*$ is upper semi-continuous at $x.$
 \end{itemize}
\end{theorem}

As Examples~\ref{exa1} and \ref{exa:} demonstrate,
Theorem~\ref{th:BGL} and Theorem~\ref{th: locprop1BS} do not imply
each other. Theorem~\ref{c: BergeLoc123} below generalizes the
both theorems. In order to clarify the relevance between
Theorems~\ref{th:BGL} and \ref{th: locprop1BS}, similarly to
Theorem~\ref{th:BGL}, it is natural to divide Theorem~\ref{th:
locprop1BS} into the following three statements:
\begin{itemize}
\item[{\bf (BS1)}] (Lower semi-continuity of minima).  If the function $u:\X \times \Y \to {\mathbb{R}}$ is lower semi-continuous on $\X \times \Y$, $x\in\X$, and assumptions (ii) and (iii) of Theorem~\ref{th: locprop1BS} hold, then
the value function $v$ is lower semi-continuous at $x;$ Bonnans and Shapiro \cite[pp.~290--291]{Shapiro}. \item[\textbf{(BS2)}] (Upper semi-continuity of minima). If a function $u:\X \times \Y \to {\mathbb{R}}$ is
upper semi-continuous on $\X\times\Y$, $x\in\X$, the set $\Phi^*(x)$ is compact, and assumption (iv) of Theorem~\ref{th: locprop1BS} holds, then the value function $v$ is upper semi-continuous at $x;$ Bonnans and Shapiro \cite[p.~291]{Shapiro}.
\item[\textbf{(BS3)}]  (Upper semi-continuity of the solution multifunction).  Let $u: \X \times \Y \to {\mathbb{R}}$ and $x\in \X.$ If assumptions (i)--(iv) of Theorem~\ref{th: locprop1BS} hold, then the solution multifunction
$\Phi^*:\X\to\Ss(\Y)$ is
upper semi-continuous at $x;$ Bonnans and Shapiro \cite[pp.~291]{Shapiro}.
\end{itemize}

Statement (BS1) can be derived from statement (B1) in the following
way. Let assumptions of statement (BS1) hold. Consider a new
Hausdorff state space $\tilde{\X}$ that equals to the neighborhood
of $x$ from assumption (iii) of Theorem~\ref{th: locprop1BS} endowed
with the induced topology.  Let $\tilde{\Y}=\Y.$ Consider new image
sets $\tilde{\Phi}(z)=\mathcal{D}_{u(z,\cdot)}(\lambda;\Phi(z))$,
$z\in \tilde{\X}$, which are nonempty and contained to the same compact
subset $C$ of $\tilde{\Y}$. Define the function $\tilde u$ as the
restriction of $u$ to ${\rm Gr}(\tilde{\Phi})$. Assumptions (ii) and (iii) of Theorem~\ref{th: locprop1BS}, Bonnans
and Shapiro \cite[Lemma~4.3]{Shapiro}, lower semi-continuity of
$\tilde u$ on $\tilde\X \times \tilde\Y$, and Feinberg et al.
\cite[Lemma~3.3(i)]{FKV} imply $\K\N$-inf-compactness of $\tilde u$
on ${\Gr}_{\tilde{\X}}(\tilde{\Phi})$. Therefore, statement (B1)
yields lower semi-continuity of $v$ at $x$. Statements (B2) and
(BS2) (also (B3) and (BS3) respectively) do not imply each other;
see Examples~\ref{exa2}--\ref{exa2b}.


In the rest of this section we formulate the main results of this paper.  The following theorem is more general that statement (B1), and therefore statement (BS1) follows from it.

\begin{theorem}{\rm (Lower semi-continuity of minima).}\label{th:locprop}
Let $u:{\rm Gr}_{\X}(\Phi)\subseteq \X\times\Y\to \overline{\mathbb{R}}$ and $x\in \X$. If the
function $u$ is $\K\N$-inf-compact on $\Gr_{\{x\}}(\Phi)$, then the
function $v:\X\to\overline{\mathbb{R}}$ is lower semi-continuous at
$x$ and $\Phi^*(x)$ is a nonempty compact set, if $v(x)<+\infty$,
and $\Phi^*(x)=\Phi(x)$ otherwise.
\end{theorem}

Observe that under conditions  of Theorem~\ref{th:locprop}, the
infimum in (\ref{eq1}) can be replaced with the minimum.  The following theorem generalizes statements (B2) and (BS2).

\begin{theorem}{\rm (Upper semi-continuity of minima).}\label{th:locUSC111}
Let $u:{\rm Gr}_{\X}(\Phi)\subseteq \X\times\Y\to \overline{\mathbb{R}}$ and $x\in \X$. Each of
the following assumptions:
\begin{itemize}
\item[(i)] the function $u:{\rm Gr}_{\X}(\Phi)\subseteq \X\times\Y\to \overline{\mathbb{R}}$ is
upper semi-continuous at all $(x,y)\in \Gr_{\{x\}}(\Phi)$ and $\Phi:
\X \to \Ss(\Y)$ is a lower semi-continuous set-valued mapping at
$x$;
\item[(ii)] $\Phi^*(x)\in \K(\Y)$, the function $u:{\rm Gr}_{\X}(\Phi)\subseteq \X\times\Y\to \overline{\mathbb{R}}$ is
upper semi-continuous at all $(x,y)\in \Gr_{\{x\}}(\Phi^*)$, and
assumption (iv) of Theorem~\ref{th: locprop1BS} holds;
\end{itemize}
implies that the function
$v:\X\to\overline{\mathbb{R}}$ is upper semi-continuous at $x$.
\end{theorem}

According to Example~\ref{exa2}, assumptions (i) and (ii) of Theorem~\ref{th:locUSC111} do not imply each other. Assumption (i) implies assumption (ii) in Theorem~\ref{th:locUSC111}, if $\Phi^*(x)\in \K(\Y)$ because lower semi-continuity of $\Phi$ at $x$ with $\Phi^*(x)\in\K(\Y)$ yields assumption (iv) of Theorem~\ref{th: locprop1BS}. Moreover, if  $\Phi^*(x)\in \K(\Y)$, then assumption (iv) of Theorem~\ref{th:locUSC111}  follows from
 upper semi-continuity as well as from lower semi-continuity of $\Phi^*$ at $x$.  The following theorem generalizes statement~(B3).

\begin{theorem} {\rm (Upper semi-continuity of the solution multifunction).} \label{th:usc}
Let  $u:{\rm Gr}_{\X}(\Phi)\subseteq \X\times\Y\to \overline{\mathbb{R}}$ and $x\in \X$. If $u$
is $\K\N$-inf-compact on $\graph_{\{x\}}(\Phi)$, the value function
$v:\X\to \overline{\mathbb{R}}$ is continuous at $x$, and
$v(x)<+\infty$, then $\Phi^*(x)\in\K(\Y)$ and $\Phi^*$ is upper
semi-continuous at $x$.
\end{theorem}

Being combined, Theorems~\ref{th:locprop} and \ref{th:locUSC111} provide
 sufficient conditions for the continuity of $v$ at $x$. Note that $\K\N$-inf-compactness of $u$ on $\graph_{\{x\}}(\Phi)$ in Theorem~\ref{th:usc} cannot be weakened to lower semi-continuity of $u$ on ${\rm Gr}_{\X}(\Phi)$, inf-compactness of $u(x,\cdot)$ on $\Phi(x)$, and the assumptions that $\Phi:\X\to\K(\Y)$ and $\Gr_\X(\Phi)$ is closed   in $\X\times\Y;$ see Example~\ref{exa3}. We also remark  that the continuity  of $v$ at $x$ is the essential assumption in Theorem~\ref{th:usc}; see Example~\ref{exa4}.

\begin{theorem} {\rm(Local optimum theorem).} \label{thm: BergeLoc}
Let  $u:{\rm Gr}_{\X}(\Phi)\subseteq \X\times\Y\to \overline{\mathbb{R}}$ and $x\in \X$ satisfy
the following properties:
\begin{itemize}
\item[{(a)}] $u(x,y)<+\infty$ for some $y\in \Phi(x)$;
\item[{(b)}] $u$ is $\K\N$-inf-compact on $\Gr_{\{x\}}(\Phi)$;
\end{itemize}
then $\Phi^*(x)\in\K(\Y).$  If, in addition,
 $u:{\rm Gr}_{\X}(\Phi)\subseteq \X\times\Y\to \overline{\mathbb{R}}$ is upper semi-continuous at all $(x,y)\in{\rm Gr}_{\{x\}}(\Phi^*)$,
then the following two assumptions are equivalent:
\begin{itemize}
\item[{\rm(i)}] assumption (iv) of Theorem~\ref{th: locprop1BS};
\item[{\rm(ii)}] the solution multifunction $\Phi^*:\X\to 2^\Y$ is
upper semi-continuous at $x$;
\end{itemize}
and each of them implies that $v$ is continuous at $x$.
\end{theorem}

Observe that assumptions of Theorem~\ref{thm: BergeLoc} include conditions on
$y^*\in \Phi(x^*)$, when $u(x^*,y^*)\ge \lambda > v(x)$, $x^*\in \oo(x)$, where $\oo(x)$ is some neighborhood of $x$. These values of $y^*$ do not affect the properties of $v$ and $\Phi^*$ at $x$. In order to obtain a more delicate result, we introduce the sets $\Phi_{\lambda,x}$ and function~$u_{\lambda,x}$.

For $\lambda\in\R$, $x\in\X$, and $\Phi:\X\to\Ss(\Y),$ define
$\Phi_{\lambda,x}:\X\to 2^\Y,$
\[
\Phi_{\lambda,x}(z):=\begin{cases} \{y\in\Phi(z)\,:\, u(z,y)\le
\lambda\}, &\mbox{if either this set is not empty or $z= x$;} \\
\Phi_{\lambda,x}(x), & {\rm otherwise.}\end{cases}\]  
As follows from this definition, $\Phi_{\lambda,x}(x)=\emptyset$ if and only if $u(x,y)>\lambda$ for all $y\in\Phi(x).$ In particular, if $\lambda<v(x)$, then $\Phi_{\lambda,x}(x)=\emptyset.$   
We remark that $\Phi_{\lambda,x}:\X\to\Ss(\Y)$ if and only if
$\Phi_{\lambda,x}(x)\ne\emptyset$. If
$\Phi_{\lambda,x}(x)\ne\emptyset$, define the functions
$u_{\lambda,x}:\X\times\Y\to \R,$
\[
u_{\lambda,x}(z,y):=\begin{cases} u(z,y), &\mbox{if }\{y\in\Phi(z)\,:\,
u(z,y)\le \lambda\}\ne\emptyset;\\
u(x,y), &\mbox{otherwise};\end{cases}
\]
for each $z\in \X$ and $y\in \Y$, and $v_{\lambda,x}:\X\to\R,$
\[
v_{\lambda,x}(z):=\inf\limits_{y\in {\Phi_{\lambda,x}}(z)}u_{\lambda,x}(z,y),\qquad 
z\in\X.
\]
Consider the solution multifunction
\begin{equation*}
{\Phi}^*_{\lambda,x}(z):=\left\{y\in
{\Phi_{\lambda,x}}(z):\,v_{\lambda,x}(z)=u_{\lambda,x}(z,y)\right\},\quad z\in\X.
\end{equation*}
 If $\lambda>v(x)$, then $v_{\lambda,x}(x)=v(x)$ and
$\Phi^*_{\lambda,x}(x)=\Phi^*(x)$.  If the
conditions of Theorem~\ref{th:locprop} hold, the latter is true
for $\lambda\ge v(x).$
Observe that, under condition (iii) of Theorem~\ref{th:
locprop1BS}, $v_{\lambda,x}(x^*)=v(x^*)$ and
$\Phi_{\lambda,x}^*(x^*)=\Phi^*(x^*)$ for all $x^*$ in a neighborhood of
$x$.

We remark that, if $u$ is $\K\N$-inf-compact on $\Gr_{\{x\}}(\Phi)$,
then $u_{\lambda,x}$ is $\K\N$-inf-compact on
$\Gr_{\{x\}}(\Phi_{\lambda,x})$. The inverse claim does not hold in
general; see Example~\ref{exa1}. The following corollary from
Theorem~\ref{th:locprop} generalizes
 statements (B1) and (BS1).

\begin{corollary}{\rm (Lower semi-continuity of minima).}\label{cor:locprop}
Let $u:{\rm Gr}_{\X}(\Phi)\subseteq \X\times\Y\to \overline{\mathbb{R}}$, $x\in\X$, and
$\lambda\in\R$ satisfy $u(x,y)\le\lambda$ for some $y\in\Phi(x).$ If
the function $u_{\lambda,x}:\Gr_\X(\Phi_{\lambda,x})\subseteq\X\times\Y\to
\overline{\mathbb{R}}$ is $\K\N$-inf-compact on
$\Gr_{\{x\}}(\Phi_{\lambda,x})$, then the function
$v:\X\to\overline{\mathbb{R}}$ is lower semi-continuous at $x$ and
$\Phi^*(x)\in\K(\Y)$.
\end{corollary}

The following theorem generalizes  Theorems~\ref{th:BGL}, \ref{th: locprop1BS}, and \ref{thm: BergeLoc}. 

\begin{theorem}\label{c: BergeLoc123}{\rm (Local optimum theorem).}
Let $u:{\rm Gr}_{\X}(\Phi)\subseteq \X\times\Y\to \overline{\mathbb{R}}$, $x\in\X$, and
$\lambda\in\R$ satisfy $u(x,y)<\lambda$ for some $y\in\Phi(x).$ If
the function $u_{\lambda,x}:\Gr_\X(\Phi_{\lambda,x})\subseteq\X\times\Y\to
\overline{\mathbb{R}}$ is $\K\N$-inf-compact on
$\Gr_{\{x\}}(\Phi_{\lambda,x})$ and the function $u:{\rm Gr}_{\X}(\Phi)\subseteq \X\times\Y\to \overline{\mathbb{R}}$ is upper semi-continuous at all $(x,y)\in{\rm
Gr}_{\{x\}}(\Phi^*)$, then the following two assumptions are
equivalent:
\begin{itemize}
\item[(i)] assumption (iv) of Theorem~\ref{th: locprop1BS};
\item[(ii)] the solution multifunction $\Phi^*:\X\to 2^\Y$ is
upper semi-continuous at $x$;
\end{itemize}
and each of them implies that $v$ is continuous at $x$.
\end{theorem}
\begin{remark}
Upper semi-continuity of $u:{\rm Gr}_{\X}(\Phi)\subseteq \X\times\Y\to \overline{\mathbb{R}}$
at all $(x,y)\in{\rm Gr}_{\{x\}}(\Phi^*)$ in Theorem~\ref{c:
BergeLoc123} cannot be relaxed to upper semi-continuity of
$u_{\lambda,x}:\Gr_\X(\Phi_{\lambda,x})\subseteq\X\times\Y\to \overline{\mathbb{R}}$ at
all $(x,y)\in{\rm Gr}_{\{x\}}(\Phi^*)$; see Example~\ref{exa:last}.
\end{remark}
\begin{remark}\label{rem4}
Theorems~\ref{th: locprop1BS}, \ref{th:locUSC111}, \ref{thm: BergeLoc}, and \ref{c: BergeLoc123} have the common
assumption: condition (iv) of Theorem~\ref{th: locprop1BS}.
 The
remaining assumptions are weaker in Theorem~\ref{c: BergeLoc123}
than in Theorem~\ref{th: locprop1BS}. Indeed, conditions (ii) and
(iii) of Theorem~\ref{th: locprop1BS} imply that $\Phi_{\lambda,x}$
is a closed mapping that acts from a closed neighborhood of $x$, say
$\overline{\oo}(x)$, into the compact set $C$. Therefore, Berge
\cite[Corollary, p.~112]{Berge} yields that
$\Phi_{\lambda,x}:\overline{\oo}(x)\to \Ss(\Y)$ is upper
semi-continuous, compact-valued, $v_{\lambda,x}(x^*)=v(x^*)$ and
$\Phi_{\lambda,x}^*(x^*)=\Phi^*(x^*)$ for all $x^*$ in a
neighborhood of $x$. Moreover, condition (i) of Theorem~\ref{th:
locprop1BS} and Lemma~\ref{lm0} imply that the function
$u_{\lambda,x}:\Gr_{\overline{\oo}(x)}(\Phi_{\lambda,x})\subseteq\X\times\Y\to
\overline{\mathbb{R}}$ is $\K\N$-inf-compact on
$\Gr_{\{x\}}(\Phi_{\lambda,x})$ and the function
$u:\Gr_{\overline{\oo}(x)}(\Phi)\subseteq\X\times\Y\to \overline{\mathbb{R}}$ is upper
semi-continuous at all $(x,y)\in{\rm Gr}_{\{x\}}(\Phi^*)$.
\end{remark}

\section{Local Properties of $\K\N$-inf-compact Functions}


\begin{lemma}\label{lm0}
Let $Z\subseteq\X$ be a nonempty set. If $u:{\rm Gr}_{\X}(\Phi)\subseteq \X\times\Y\to \overline{\mathbb{R}}$ is lower
semi-continuous at all $(x,y)\in{\rm
Gr}_{Z}(\Phi)$ and ${\Phi}:\X\to
\K(\Y)$ is upper semi-continuous at all $x\in Z$, then the function
$u(\cdot,\cdot)$ is $\K\N$-inf-compact on ${\rm Gr}_{Z}
({{\Phi}})$;
\end{lemma}
\begin{proof}
The proof is similar to the proof of Feinberg at el. \cite[Lemma~3.3(i)]{FKV}.
Let $\{x_i \}_{i\in I}$ be a convergent net with values in
$\X$ whose limit $x$ belongs to $Z$ and $\{y_i\}_{i\in I}$ be a
net defined on the same ordered set $I$ with $y_i\in \Phi(x_i)$,
$i\in I,$ and satisfying the condition that the set $\{u(x_i,y_i):
i\in I \}$ is bounded above by $\lambda\in\mathbb{R}$. Let us
prove that a net $\{y_i\}_{i\in I}$ has an accumulation point
$y\in \Phi(x)$ such that $u(x,y)\le \lambda$. Aliprantis and
Border \cite[Corollary 17.17, p.~564]{AliBor} yields that a net
$\{y_i\}_{i\in I}$ has an accumulation point $y\in \Phi(x)$. The
lower semi-continuity of $u$  at all $(x,y)\in{\rm
Gr}_{Z}(\Phi)$ implies that
$u(x,y)\le \lambda$. Therefore, the function $u(\cdot,\cdot)$ is
$\K\N$-inf-compact on ${\rm Gr}_Z ({{\Phi}})$. \hfill$\Box$
\end{proof}

  When the topological spaces $\X$ and $\Y$ are metrizable, we may avoid nets in the definition of the $\K\N$-inf-compactness by replacing them with sequences.

\begin{lemma}\label{lmmetr}
Let $\X$ and $\Y$ be metrizable spaces, $Z\subseteq\X$ be a nonempty
set. Then $u:{\rm Gr}_{\X}(\Phi)\subseteq \X\times\Y\to \overline{\mathbb{R}}$ is $\K\N$-inf-compact on
${\rm Gr}_{Z}(\Phi)$ if and only if the following two conditions
hold:

(i) $u:{\rm Gr}_{\X}(\Phi)\subseteq \X\times\Y\to \overline{\mathbb{R}}$ is lower semi-continuous at
all $(x,y)\in{\rm Gr}_{Z}(\Phi)$;

(ii) if a sequence $\{x_n \}_{n=1,2,\ldots}$ with values in $\X$
converges to $x\in Z$ then each sequence
$\{y_n \}_{n=1,2,\ldots}$ with $y_n\in \Phi(x_n)$, $n=1,2,\ldots,$
satisfying the condition that the sequence $\{u(x_n,y_n)
\}_{n=1,2,\ldots}$ is bounded above, has a limit point $y\in
\Phi(x).$
\end{lemma}

\begin{proof}
Condition (i) from the definition of $\K\N$-inf-compactness coincides with assumption (i) of the lemma. The rest of the proof establishes the equivalency of assumption (ii) of the lemma and assumption (ii) from Definition~\ref{def:locF-ic}. 

Let  assumption (ii) of the lemma hold. Consider a net $\{x_i
\}_{i\in I}$ with values in $\X,$ that converges to $x\in Z$, and
a net $\{y_i\}_{i\in I}$  defined on the same ordered set $I$ with $y_i\in
\Phi(x_i)$, $i\in I,$ and
satisfying the condition that the set $\{u(x_i,y_i): i\in I \}$ is bounded
above. Let us prove that $\{y_i\}_{i\in I}$ has an accumulation point $y\in \Phi(x).$
Since the space $\X$ is metrizable, it is first-countable and each its point has a countable neighborhood basis (local base). That is, for $x \in Z$ there exists a sequence $\oo_1, \oo_2, \ldots$ of neighborhoods of $x$ such that for each neighborhood $\oo(x)$ of $x$ there exists an integer $n$ with $\oo_n$ contained in $\oo(x)$.
Since the net $\{x_i\}_{i\in I}$ converges to $x$, for each $N=1,2,\ldots,$ there exists an index  $i_N\in I$ such that $x_i\in \oo_N$ for each $i\succeq i_N$. Thus, the sequence $\{x_{i_N}\}_{N=1,2,\ldots}$ converges to $x$, and the sequence $\{y_{i_N}\}_{N=1,2,\ldots},$ with $y_{i_N}\in\Phi(x_{i_N})$, $N=1,2,\ldots$, satisfies the condition that the set  $\{u(x_{i_N},y_{i_N})\,:\, N=1,2,\ldots \}$ is bounded
above. In view of assumption (ii) of the lemma, the sequence $\{y_{i_N}\}_{N=1,2,\ldots}$ has a limit point $y\in\Phi(x)$. Therefore, $y\in \Phi(x)$ is the accumulation point of the net $\{y_i\}_{i\in I}$.

Let assumption (ii) of Definition~\ref{def:locF-ic} hold. Consider a sequence $\{x_n \}_{n=1,2,\ldots}$ with values in $\X,$ that
converges to $x\in Z,$ and a sequence
$\{y_n \}_{n=1,2,\ldots}$ with $y_n\in \Phi(x_n)$, $n=1,2,\ldots,$
such that the sequence $\{u(x_n,y_n)
\}_{n=1,2,\ldots}$ is bounded above. Let us prove that the sequence $\{y_n \}_{n=1,2,\ldots}$ has a limit point $y\in
\Phi(x).$ Let $C$ be a closure of the set $\{(x_n,y_n)\,:\,n=1,2,\ldots\}$ in $\X\times\Y$. Condition (ii) of Definition~\ref{def:locF-ic} yields that $C$ is a compact set and $C\subseteq \{(x_n,y_n)\,:\,n=1,2,\ldots\}\cup {\rm Gr}_{\{x\}}(\Phi)$. Since the space $\X\times\Y$ is metrizable, the sequence $\{(x_n,y_n)
\}_{n=1,2,\ldots}\subset C$ has a convergent subsequence $\{(x_{n_k},y_{n_k})\}_{k=1,2,\ldots}$ to  $(x,y)\in C$. Since $C\subset \Gr_\X(\Phi)$, then $(x,y)\in{\rm Gr}_{\{x\}}(\Phi)\subseteq {\rm Gr}_{Z}(\Phi)$. Therefore, the sequence $\{y_n \}_{n=1,2,\ldots}$ has a limit point $y\in
\Phi(x).$  \hfill$\Box$
\end{proof}

\section{Proofs of the Main Results}

This section contains the proofs of Theorems~\ref{th:locprop}--\ref{c: BergeLoc123} and Corollary~\ref{cor:locprop}.

\textit{Proof of Theorem~\ref{th:locprop}.} Let $x\in\X$, and the
function $u:{\rm Gr}_{\X}(\Phi)\subseteq \X\times\Y\to \overline{\mathbb{R}}$ be
$\K\N$-inf-compact on $\Gr_{\{x\}}(\Phi)$. If $v(x)=+\infty$, then
$\Phi^*(x)=\Phi(x)$. Otherwise, $\Phi^*(x)$ is a nonempty compact
set; cf. Feinberg et al. \cite[Theorem~3.1]{FKV}. Let us prove that the
function $v:\X\to\overline{\mathbb{R}}$ is lower semi-continuous at
$x$.  If $v(x)=-\infty$, then $v$ is lower semi-continuous at $x$.
Let $v(x)>-\infty$. Consider a net $\{x_i \}_{i\in I}$ with values
in $\X$ converging to $x$. Choose a subnet $\{x_j \}_{j\in J}$ of
the net $\{x_i \}_{i\in I}$ such that $\ilim_i v(x_i)=\lim_j
v(x_j)$. There are two alternatives: either $\{v(x_j)\}_{j\in J}$
converges to $+\infty$, or $\{v(x_j)\}_{j\in J}$ is bounded above.
In the first case $v(x)\le \ilim_i v(x_i)=+\infty$, that is, $v$ is
lower semi-continuous at $x$. If the second alternative holds, then
for each $j\in J$ there exists $y_j\in \Phi(x_j)$ such that
$\{|v(x_j)-u(x_j,y_j)|\}_{j\in J}$ converges to zero and, therefore,
the net $\{u(x_j,y_j)\}_{j\in J}$ is bounded above. Condition (ii)
of the definition of $\K\N$-inf-compactness on $\Gr_{\{x\}}(\Phi)$
implies that the net $\{y_j\}_{j\in J}$ has an accumulation point
$y\in \Phi(x)$. Condition (i) of the definition of
$\K\N$-inf-compactness on $\Gr_{\{x\}}(\Phi)$ yields that $\lim_j
v(x_j)\ge u(x,y)\ge v(x)$. Since $\ilim_i v(x_i)=\lim_j v(x_j)$, the
function $v$ is lower semi-continuous at $x$.
  \hfill $\Box$


\textit{Proof of Theorem~\ref{th:locUSC111}.} Let $u:{\rm Gr}_{\X}(\Phi)\subseteq \X\times\Y\to \overline{\mathbb{R}}$, $x\in \X,$ and a net
$\{x_\alpha\}_{\alpha\in I}$ with values in $\X$ converge to $x$. If
$v(x)=+\infty$, then $v$ is obviously upper semi-continuous at $x$.
Consider the case $v(x)<+\infty$.

If assumption (i) holds, then, in view of the lower semi-continuity
of $\Phi$ at $x$, for each $y\in \Phi(x)$ and each $\alpha\in I$
there exists $y_\alpha\in\Phi(x_\alpha)$ such that the net
$\{y_\alpha\}_{\alpha \in I}$ converges to $y$ in $\Y$. Therefore,
due to upper semi-continuity of $u:{\rm Gr}_{\X}(\Phi)\subseteq \X\times\Y\to \overline{\mathbb{R}}$ at
each $(x,y)\in \Gr_{\{x\}}(\Phi)$,
\[\slim_\alpha v(x_\alpha)\le  \slim_\alpha u(x_\alpha,y_\alpha)\le u(x,y)\mbox{ for each }y\in\Phi(x),
\] that is,
\[
\slim_\alpha v(x_\alpha)\le \inf_{y\in \Phi(x)}u(x,y)=v(x),
\]
and $v$ is upper semi-continuous at $x$; Hu and Papageorgiou \cite[Proposition~3.1, p.~82]{Hu}.

Let assumption (ii) holds. Denote by $\tau_\X$ and $\tau_\Y$ the
topologies (the families of all open subsets) of $\X$ and $\Y$
respectively. Note that the family of open sets
$\tau_{\X\times\Y}^b:=\{\oo_\X\times\oo_\Y\,:\, \oo_\X\in \tau_\X,\,
\oo_\Y\in\tau_\Y \}$ is a base of the topology of the Hausdorff
topological product space $\X\times\Y$ (that is, each open subset of
$\X\times\Y$ can be presented as a union of some sets from
$\tau_{\X\times\Y}^b$).  This is true because the product topology (so-called
natural or Hausdorff topology) on the Cartesian  product of a finite
number of Hausdorff spaces coincides with the box topology.

If for any $\lambda>v(x)$ there is a neighborhood of $x$, $\oo(x)\in \tau_\X$, such that
\begin{equation}\label{eq:ii2_1}
v(x^*)< \lambda \mbox{ for each }x^*\in \oo(x),
\end{equation}
then the function $v$ is upper semi-continuous at $x$. Fix an arbitrary $\lambda>v(x)$. The rest of the proof establishes the existence of $\oo(x)\in\tau_\X$ satisfying (\ref{eq:ii2_1}).

In view of the upper semi-continuity of $u:{\rm Gr}_{\X}(\Phi)\subseteq \X\times\Y\to \overline{\mathbb{R}}$ at each $(x,y)\in\Gr_{\{x\}}(\Phi^*)$, for each $y\in \Phi^*(x)$ there exists $\oo_\X^{x,y}(x)\times\oo_\Y^{x,y}(y)\in \tau_{\X\times\Y}^b$ such that
\begin{equation}\label{eq:ii2_2}
u(x^*,y^*)<\lambda\mbox{ for all }(x^*,y^*)\in \left(\oo_\X^{x,y}(x)\times\oo_\Y^{x,y}(y)\right)\cap \Gr_\X(\Phi)
\end{equation}
because $\lambda>v(x)=u(x,y)$ for any $y\in \Phi^*(x)$.
The collection of open sets $\{\oo_\X^{x,y}(x)\times\oo_\Y^{x,y}(y)\,:\, y\in \Phi^*(x)\}$ covers the set $\Gr_{\{x\}}(\Phi^*)$. Tychonoff's theorem yields that the set $\Gr_{\{x\}}(\Phi^*)\subset \X\times\Y$
is compact. Therefore, there is a finite cover $\{\oo_\X^{x,y_1}(x)\times\oo_\Y^{x,y_1}(y_1),\oo_\X^{x,y_2}(x)\times\oo_\Y^{x,y_2}(y_2),\ldots, \oo_\X^{x,y_N}(x)\times\oo_\Y^{x,y_N}(y_N)\}$, $y_1,y_2\ldots,y_N\in \Phi^*(x)$, $N=1,2,\ldots,$ of $\Gr_{\{x\}}(\Phi^*)$. Since $x\in \oo_\X^{x,y_j}(x)$ for each $j=1,2,\ldots,N$, the family of sets $\{\oo_\X(x)\times\oo_\Y^{x,y_1}(y_1),\oo_\X(x)\times\oo_\Y^{x,y_2}(y_2),\ldots, \oo_\X(x)\times\oo_\Y^{x,y_N}(y_N)\}$, where $\oo_\X(x)=\cap_{j=1}^N \oo_\X^{x,y_j}(x)\in \tau_\X$, covers the set $\Gr_{\{x\}}(\Phi^*)$, that is,
\begin{equation}\label{eq:ii2_3}
\Gr_{\{x\}}(\Phi^*)\subseteq \oo_\X(x)\times\cup_{j=1}^N \oo_\Y^{x,y_j}(y_j)\subseteq \cup_{y\in \Phi^*(x)} \oo_\X^{x,y}(x)\times\oo_\Y^{x,y}(y).
\end{equation}
In particular,
$\oo_\Y(\Phi^*(x)):=\cup_{j=1}^N\oo_\Y^{x,y}(y_j)\in \tau_\Y$ is a neighborhood of $\Phi^*(x)$.
Assumption (ii) implies that for the neighborhood $\oo_\Y(\Phi^*(x))$ of $\Phi^*(x)$ there exists a neighborhood $\oo_\X^1(x)$ of $x$ such
that $\oo_\Y(\Phi^*(x))\cap \Phi(x^*)\ne \emptyset$ for all $x^*\in
\oo_\X^1(x)$. Therefore, according to (\ref{eq:ii2_2}) and (\ref{eq:ii2_3}),
\[
v(x^*)\le\inf_{y^*\in \oo_\Y(\Phi^*(x))\cap \Phi(x^*) }u(x^*,y^*) <\lambda ,
\]
for each $x^*\in \oo(x):=\oo_\X(x)\cap \oo_\X^1(x)\in \tau_\X$. Thus  (\ref{eq:ii2_1}) holds.
   \hfill$\Box$

%
%


\textit{Proof of Theorem~\ref{th:usc}.} Let $u:{\rm Gr}_{\X}(\Phi)\subseteq \X\times\Y\to \overline{\mathbb{R}}$, $x\in\X$, and let $v: \X\to \overline{\R}$ be the
value function defined in (\ref{eq1}). If $u$ is $\K\N$-inf-compact
on $\graph_{\{x\}}(\Phi)$, $v$ is continuous at $x$, and
$v(x)<+\infty$, then Theorem~\ref{th:locprop} yields that
$\Phi^*(x)$ is a nonempty compact set. Let us prove that $\Phi^*$ is
upper semi-continuous at $x$. Suppose, on the contrary, that
$\Phi^*$ is not upper semi-continuous at $x\in \X$. Then there is an
open neighborhood $\oo(\Phi^*(x))$ of $\Phi^*(x)$ such that, for
every neighborhood $\oo$ of $x$, there is an $x_\oo \in \oo$ with
$\Phi^*(x_\oo) \not\subseteq \oo(\Phi^*(x))$. Therefore, there
exists $y_\oo\in \Phi^*(x_\oo) \setminus \oo(\Phi^*(x))$. Now
consider the nets $\{x_\oo: \oo \in \tau_b^{\rm loc}(x)\}$ and
$\{y_\oo: \oo \in \tau_b^{\rm loc}(x)\}$, where $\tau_b^{\rm
loc}(x)$ is the directed set of neighborhoods of $x$; see  Zgurovsky
et al.~\cite[p.~9]{ZMK}  for details. The net $\{x_\oo\,:\, \oo \in
\tau_b^{\rm loc}(x)\}$ converges to $x.$ Then each net
$\{y_\oo\,:\,\oo\in \tau_b^{\rm loc}(x)\},$ defined on the same
ordered set $\tau_b^{\rm loc}(x)$ with $y_\oo\in \Phi^*(x_\oo)$,
$\oo\in \tau_b^{\rm loc}(x),$ has an accumulation point $y\in
\Phi^*(x).$  Indeed, $v(x_\oo)=u(x_\oo,y_\oo)$ for each $\oo\in
\tau_b^{\rm loc}(x)$. Since $x_\oo\to x$ and $v$ is continuous at
$x$, the net $\{v(x_\oo)\,:\,\oo\in \tau_b^{\rm loc}(x)\}$ is
bounded above by a finite constant eventually in $\tau_b^{\rm
loc}(x)$. Therefore, $\K\N$-inf-compactness of the function $u$ on
$\Gr_{\{x\}}(\Phi)$ implies that the net $\{y_\oo\,:\,\oo\in
\tau_b^{\rm loc}(x)\}$ has an accumulation point
$y\in \Phi(x).$ Since $u:{\rm Gr}_{\X}(\Phi)\subseteq \X\times\Y\to \overline{\mathbb{R}}$ is lower semi-continuous at each $(x,y)\in\Gr_{\{x\}}(\Phi)$ and $v$ is  continuous at $x$, 
then $y\in\Phi^*(x)$, that is, the net
$\{y_\oo\,:\,\oo\in \tau_b^{\rm loc}(x) \}$ has an accumulation point $y \in \Phi^*(x)\subseteq
\oo(\Phi^*(x))$. The net $\{y_\oo\,:\,\oo\in \tau_b^{\rm loc}(x)\}$ lies in the closed set $\oo(\Phi^*(x))^c$, which is the
complement of $\oo(\Phi^*(x))$, and thus $y\in \oo(\Phi^*(x))^c$. This contradiction
implies that $\Phi^*$ is upper semi-continuous~at~$x$.
$\ $   \hfill $\Box$
%

\textit{Proof of Theorem~\ref{thm: BergeLoc}.}
According to Theorem~\ref{th:locprop}, the function
$v:\X\to\overline{\mathbb{R}}$ is lower semi-continuous at $x$ and $\Phi^*(x)$ is a nonempty compact set. If assumptions (i) and (ii) of Theorem~\ref{thm: BergeLoc} imply each other, then
Theorem~\ref{th:locUSC111}(ii) yields that the value function $v$ is 
upper semi-continuous at $x$.
Therefore, to finish the proof, it is sufficient to prove that assumptions (i) and (ii) of Theorem~\ref{thm: BergeLoc} are equivalent.

The implication ``$(i) \Longrightarrow (ii)$'' directly follows from Theorems~\ref{th:locprop}, \ref{th:locUSC111}(ii), and~\ref{th:usc}. Vice versa, if the solution multifunction $\Phi^*:\X\to 2^\Y$  is
upper semi-continuous at $x$, then for
each neighborhood $\oo(\Phi^*(x))$ of the set $\Phi^*(x)$, there is a
neighborhood of $x$, say $\oo(x)$, such that
$\Phi^*(x^*)\subseteq \oo(\Phi^*(x))$ for all $x^*\in \mathcal{O}(x)$. Thus,
for each neighborhood $\oo(\Phi^*(x))$ of
the set $\Phi^*(x)$ there exists a neighborhood $\oo(x)$ of $x$ such
that $\oo(\Phi^*(x))\cap \Phi(x^*)\ne \emptyset$ for all $x^*\in
\oo(x)$. Therefore, assumptions (i) and (ii) of Theorem~\ref{thm: BergeLoc} are equivalent. \hfill $\Box$


\textit{Proof of Corollary~\ref{cor:locprop}.}
According to Theorem~\ref{th:locprop},  the function
$v_{\lambda,x}:\X\to\overline{\mathbb{R}}$ is lower semi-continuous at $x$ and $\Phi_{\lambda,x}^*(x)$ is a nonempty compact set because $v_{\lambda,x}(x)\le\lambda<+\infty$. Since $u(x,y)\le\lambda$ for some
$y\in\Phi(x)$, then $y\in\Phi^*(x)$ and $v(x)=v_{\lambda,x}(x)$. Moreover, if $\Phi_{\lambda,x}^*(x^*)\ne\emptyset$ for some $x^*\in \X$, then $v_{\lambda,x}(x^*)=v(x^*)$; otherwise, $v(x^*)\ge \lambda\ge v(x)$. Therefore, the function
$v:\X\to\overline{\mathbb{R}}$ is lower semi-continuous at $x$ and $\Phi^*(x)$ is a nonempty compact set. \hfill$\Box$

\textit{Proof of Theorem~\ref{c: BergeLoc123}.} According to
Corollary~\ref{cor:locprop}, the value function $v$ is lower
semi-continuous at $x$ and $\Phi^*(x)\in\K(\Y)$. Since
$u(x,y)<\lambda$ for some $y\in\Phi(x)$ and the function
$u:{\rm Gr}_{\X}(\Phi)\subseteq \X\times\Y\to \overline{\mathbb{R}}$ is upper semi-continuous
at all $(x,y)\in{\rm Gr}_{\{x\}}(\Phi^*)$, there exists a
neighborhood of $x$, say $\oo(x)$, such that the level set
$\mathcal{D}_{u(x^*,\cdot)}(\lambda;\Phi(x^*))$ is nonempty for each
$x^*\in \oo(x)$. Therefore, $v_{\lambda,x}(x^*)=v(x^*)$ and
$\Phi_{\lambda,x}^*(x^*)=\Phi^*(x^*)$ for each $x^*\in \oo(x)$. 
Note that assumption (ii) of Theorem~\ref{c: BergeLoc123} implies assumption (i) of Theorem~\ref{c: BergeLoc123}. Theorem~\ref{th:locUSC111} yields that
the function $v:\X\to\overline{\R}$ is upper semi-continuous at $x$.
Therefore, the function $v=v_{\lambda,x}$ is continuous at $x$.
 Being applied to $u_{\lambda,x}$, $v_{\lambda,x}$ and $\Phi_{\lambda,x},$ Theorem~\ref{th:usc}, yields that
assumptions (i) and (ii) of Theorem~\ref{c:
BergeLoc123} are equivalent.
\hfill$\Box$

\section{Examples and Counterexamples}

The following example illustrates that Theorem~\ref{th: locprop1BS} can be applied to a function $u:\X\times\Y\to\R,$ which is not $\K\N$-inf-compact on $\Gr_{\X}(\Phi)$.

\begin{example}\label{exa1}
{\rm  Let $\X=\Y=\R$, $u(x^*,y^*)=\min\{|x^*-y^*|,1\}$, $\Phi(x^*)=\R$, $x^*,y^*\in \R$. Then $u$ is neither $\K$-inf-compact nor $\K\N$-inf-compact on $\R^2$, because\newline $\mathcal{D}_{u}(1;\Gr_{\{0\}}(\Phi))=\Y\notin\K(\Y)$.
Thus, the assumptions of Theorem~\ref{th:BGL} do not hold, but the assumptions of Theorem~\ref{th: locprop1BS}
are obviously hold for $x=0$, $\lambda=\frac12$, and $\Phi^*(0)=\{0\}$.}
\hfill $\Box$\end{example}

In the following example the assumptions of Theorem~\ref{th:BGL} hold, but the assumptions of Theorem~\ref{th: locprop1BS} do not hold.
\begin{example}\label{exa:}
{\rm
Let $\X=\Y=l_1:=\{(a_1,a_2,\ldots)\,:\, |a_1|+|a_2|+\ldots <\infty\}$ be metric spaces with the metric
$\rho(a,b):=\|a-b\|=|a_1-b_1|+|a_2-b_2|+\ldots$ for each $a=(a_1,a_2,\ldots), \,b=(b_1,b_2,\ldots)\in l_1$.
Denote $\bar{0}:=(0,0,\ldots)$ and $ \bar{B}_\delta(\bar{0})=\{a=(a_1,a_2,\ldots)\in l_1\,:\, |a_1|+|a_2|+\ldots \le \delta \}$, where $\delta>0.$  The balls
$ \bar{B}_\delta(\bar{0})$ are not compact sets because the sequence $\{a^{(n)}\}_{n=1,2,\ldots},$ with $a^{(n)}_n=\delta$ and $a^{(n)}_i=0$ for $i\ne n,$ does not contain a convergent subsequence.


Let $\Phi(x^*)=\{x^*\}$  and  $u(x^*,y^*)=\frac12(\|x^*\|+\|y^*\|)$ for all $(x^*,y^*)\in\X\times \Y$.  Fix $x=\bar{0}.$
Note that $u$ is continuous on $\X\times\Y$ and $\K\N$-inf-compact on $\Gr_\X(\Phi)=\{(z,z)\,:\, z\in \X\}$. The set-valued mapping $\Phi:\X\to \K(\Y)$ is continuous. Therefore, the assumptions of Theorem~\ref{th:BGL} hold.

Assumption (iii) of Theorem~\ref{th: locprop1BS} does not hold. Indeed, on the contrary, if there
exist $\lambda\in \R$ and a compact set $C\subset \Y$ such that, for
every $x^*$ in a neighborhood of $\bar{0},$ say $\oo(\bar{0})$, the level set
$\mathcal{D}_{u(x^*,\cdot)}(\lambda;\Phi(x^*))$ is nonempty and
contained in $C$, then there exists $\delta>0$ such that $\oo(\bar{0})\supseteq \bar{B}_\delta(\bar{0}),$ $\delta\le \lambda$, and $\bar{B}_\delta(\bar{0})\subseteq C$. Since the closed subset $\bar{B}_\delta(\bar{0})$ of a compact set $C$ is compact, we obtain a contradiction, because the ball $\bar{B}_\delta(\bar{0})$ is not compact. 
}
\hfill $\Box$\end{example}

The following example demonstrates that (a): assumptions (i) and (ii) of Theorem~\ref{th:locUSC111} do not imply each other; and (b) statements (B2) and (BS2) do not imply each other.
\begin{example}\label{exa2}
{\rm Set $\X=[0,1]$, $\Y=[-1,0]$, $u(x^*,y^*)=|x^*-y^*|$, $(x^*,y^*)\in\X\times\Y$, and $x=0$. Note that $u$ is finite and continuous on $\X\times\Y$.

Let  $\Phi(x^*)=[-1,0)$, $x^*\in\X.$ Then $\Phi: \X \to \Ss(\Y)$ is a lower semi-continuous set-valued mapping. Thus assumption (i) of Theorem~\ref{th:locUSC111} and the assumptions of statement (B2) hold.  However, neither assumption (ii) of Theorem~\ref{th:locUSC111} nor the assumptions of statement (BS2) 
hold, because $\Phi(x^*)=\emptyset$ for each $x^*\in \X$.

Now let $\Phi(x^*)=\{0, -{\bf I}\{x^*= 0\}\}$, $x^*\in \X.$ Then
$\Phi^*(x^*)=\{0\}\in \K(\Y)$, for each $x^*\in \X$, and for each neighborhood $\oo(\Phi^*(x))$ of
the set $\Phi^*(x)$ there exists a neighborhood $\oo(x)$ of $x$ such
that $\oo(\Phi^*(x))\cap \Phi(x^*)\ne \emptyset$ for all $x^*\in
\oo(x).$ Thus assumption (ii)  of Theorem~\ref{th:locUSC111} and the assumptions of statement (BS2) hold.  However, neither assumption (i) of Theorem~\ref{th:locUSC111} nor  the assumptions of statement (B2)  hold, because $\Phi$ is not lower semi-continuous at $x$.}
\hfill $\Box$\end{example}

The following two examples yield that statements (B3) and (BS3) do not imply each other.
\begin{example}\label{exa2a}
{\rm (The assumptions of statement (BS3) hold, but the assumptions
of statement (B3) do not hold). For this purpose, consider
Example~\ref{exa1}. Note that: (a) the function $u$ is continuous on
$\X \times \Y$, (b) the set-valued mapping $\Phi$ is closed, (c) for
every $x^*$ in a neighborhood of $x=0$, say $(-\frac12,\frac12)$,
the level set
$\mathcal{D}_{u(x^*,\cdot)}(\lambda;\Phi(x^*))=[x^*-\frac12,x^*+\frac12]$
is nonempty and is contained in the compact set $C=[-1,1]$, and (d)
for each neighborhood $\oo(\Phi^*(x))$ of the set $\Phi^*(x)=\{0\}$
there exists a neighborhood $\oo(x)=(-1,1)$ of $x$ such that
$\oo(\Phi^*(x))\cap \Phi(x^*)\ne \emptyset$ for all $x^*\in \oo(x).$  The latter is true
since $\Phi\equiv\R$. Therefore, the assumptions of statement
(BS3) hold, but the assumptions of statement (B3) do not hold.}
\hfill $\Box$\end{example}

\begin{example}\label{exa2b}
{\rm (The assumptions of statement (B3) hold, but the assumptions of statement (BS3) do not hold). Let $\X=[0,1]$, $\Y=[0,2]$, $\Phi\equiv [0,2]$, $u(x^*,y^*)={\bf I}\{x^*-y^*<0\}$, $x^*\in\X$, $y^*\in\Y$. Then $u$ is $\K\N$-inf-compact on $\X\times\Y$, $v\equiv 0$ is continuous, but $u$ is discontinuous at each point $(x,x)\in [0,1]^2$.
Therefore, the assumptions of statement (B3) hold, but the assumptions of statement (BS3) do not hold.}
\hfill $\Box$\end{example}

The following example shows that $\K\N$-inf-compactness of $u$ on $\graph_{\{x\}}(\Phi)$ in Theorem~\ref{th:usc} cannot be weakened to lower semi-continuity of $u:{\rm Gr}_{\X}(\Phi)\subseteq \X\times\Y\to \overline{\mathbb{R}}$, inf-compactness of $u(x,\cdot)$ on $\Phi(x)$, and closeness of $\Gr_\X(\Phi)$ in $\X\times\Y$ with $\Phi:\X\to\K(\Y)$.
\begin{example}\label{exa3}
{\rm Let $\X=[0,1]$, $\Y=[0,+\infty)$, $u\equiv 0$, $x=0$, $\Phi(x^*)=\{1/x^*\}$ if $x^*\in(0,1]$, and $\Phi(0)=\{0\}$. Then $u$ is continuous on $\X\times\Y$, $v\equiv 0$ is continuous on $\X$, $\Phi^*(x^*)=\Phi(x^*)\in \K(\Y)$, $x^*\in \X$, $\Gr_\X(\Phi)$ is closed in $\X\times\Y$, $u(0,\cdot)\equiv 0$ is inf-compact on $\Phi^*(x)=\{0\}$, but $\Phi^*$ is not upper semi-continuous at $x=0$.}
\hfill $\Box$\end{example}

The following example demonstrates  that continuity  of $v$ at $x$ in Theorem~\ref{th:usc} is the essential assumption.

\begin{example}\label{exa4}
{\rm Let $\X=\Y=[0,1]$, $u(x^*,y^*)=y^*{\bf I}\{x^*>0\}$, $x^*,y^*\in[0,1]$, $\Phi(x^*)=\{1/x^*\}$ for $x^*\in(0,1]$, $\Phi(0)=\{0\}$, and $x=0$. Then, $u$ is $\K\N$-inf-compact on $\Gr_{\X}(\Phi)$, $v(x^*)=\frac1{x^*}$ for $x^*\in(0,1]$, $v(0)=0$, and $\Phi^*(x^*)=\Phi(x^*)$, $x^*\in \X$. Thus, $v$ is discontinuous at $x$ and $\Phi^*:\X\to \K(\Y)$ is not upper semi-continuous at $x$.
}
\hfill $\Box$\end{example}

In the following example the assumptions of
Corollary~\ref{c: BergeLoc123} hold, but the assumptions of
Theorem~\ref{th: locprop1BS} do not hold.

\begin{example}\label{ex} {\rm Let $\X=\Y=[0,1]$, $\Phi(x^*)=\{x^*\}$ for all $x^*\in[0,1]$, and
$u(x^*,y^*)=-x^*{\bf I}\{x^*\in \mathbb{Q}\}$ for $x^*,y^*\in [0,1]$, where
$\mathbb{Q}$ is the set of rational numbers. Assumptions of
Corollary~\ref{c: BergeLoc123} hold for $x=0$ and $\lambda=0$,
because $\Phi_{\lambda,x}=\Phi$ is continuous, $u_{\lambda,x}=u$ is
$\K\N$-inf-compact on ${\rm Gr_{\{0\}}}(\Phi_{\lambda,x})=\{(0,0)\}$ and
it is continuous at the point $(0,0)$. Assumption (i) of
Theorem~\ref{th: locprop1BS} does not hold, because the function $x^*\to
-x^*{\bf I}\{x^*\in \mathbb{Q}\}$ is not continuous on $[0,a)$ for each $a\in (0,1).$} 
\hfill $\Box$\end{example}

The following example demonstrates that  upper semi-continuity of $u$ at all $(x,y)\in{\rm Gr}_{\{x\}}(\Phi^*)$ in Corollary~\ref{c: BergeLoc123} cannot be relaxed to upper semi-continuity of $u_{\lambda,x}$ at all $(x,y)\in{\rm Gr}_{\{x\}}(\Phi^*)$.

\begin{example}\label{exa:last} {\rm Let $\X=\Y=[0,1]$, $\Phi(x^*)=\{x^*\}$  and
$u(x^*,y^*)={\bf I}\{x^*\ne 0\}$ for all $x^*\in [0,1]$ and $y^*=x^*,$  $\lambda=\frac12$, and $x=0$. Then  $u(0,0)<\lambda$, the function $u_{\lambda,x}\equiv0 $ is
$\K\N$-inf-compact on $\Gr_{\{x\}}(\Phi_{\lambda,x})$ and upper semi-continuous at any  $(x,y)\in \Gr_\X(\Phi)$, but the function $u$ is not upper semi-continuous at $(0,0)\in{\rm Gr}_{\{0\}}(\Phi^*)$ and the value function $v(x^*)={\bf I}\{x^*\ne 0\}$, $x^*\in[0,1],$ is not continuous at $x=0$.
}
\hfill $\Box$\end{example}

\begin{acknowledgements}
The authors thank Alexander Shapiro for bringing the local maximum theorem from Bonnans  and Shapiro~\cite[Proposition~4.4]{Shapiro}  to their attention and for stimulating discussions.
\end{acknowledgements}


\bibliographystyle{elsarticle-num}

\end{document}